\documentclass{article}
\usepackage[T1]{fontenc}

\usepackage{amsthm}
\usepackage{amsmath}
\usepackage{amssymb}
\usepackage{wrapfig}
\usepackage[lofdepth,lotdepth]{subfig}

\input{insbox}

\usepackage{hyperref}
\hypersetup{
    colorlinks=true,
    linkcolor=blue,
    filecolor=magenta,
    citecolor = red,      
    urlcolor=cyan,
    pdftitle={Overleaf Example},
    pdfpagemode=FullScreen,
    }

\usepackage{wrapfig}
\usepackage{cutwin}

\usepackage{pgf,tikz,tikz-3dplot}
\usetikzlibrary{decorations.markings,graphs, graphs.standard,calc,arrows}
\usepackage{mathrsfs}
\usetikzlibrary{arrows}

\usepackage{graphicx}%
\usepackage{multirow}

\usepackage{enumerate}
\usepackage{url}

\usepackage[a4paper, top = 20mm, left = 25mm, right = 25mm,bottom = 20mm]{geometry}

\newtheorem{theorem}{Theorem}

\newcommand{\R}{\mathbb{R}}

\title{Cutting along a symmetric quadrilateral to construct an embedded flexible dodecahedron}
\author{Elvar W. Atlason}
\date{September 2026}

\begin{document}

\maketitle

\begin{abstract}
Until recently, the simplest known flexible polyhedron embedded in Euclidean three-space was Steffen's polyhedron on nine vertices. However, in 2024, an embedded flexible polyhedron on eight vertices was discovered. It attains the known lower bound for the number of vertices, showing that the simplest embedded flexible polyhedron has eight vertices.

We introduce a method for making new flexible polyhedral surfaces from old ones. This general method applies to the above minimal example, giving another proof of its flexibility. We also construct a different, embedded flexible dodecahedron on eight vertices. This improves both the range of motion and the simplicity of the exposition.
\end{abstract}

\tableofcontents

\section*{Introduction}

Polyhedra are generically rigid, as shown in \cite{Gluck_genericity}. Embedded polyhedra were conjectured to be rigid, but this was disproved by Robert Connelly in 1977, \cite{Connelly_original}. Since then, a few examples of embedded flexible polyhedra have been found. The simplest known was Steffen's Polyhedron on nine vertices, and it was believed to be optimal, \cite{Mak95}. This turned out to be wrong, and in the paper \cite{Minimal_example} from 2024, an embedded flexible polyhedron on eight vertices was found. It attains the known lower bound for the number of vertices of a flexible polyhedron, showing that the simplest embedded flexible polyhedron has eight vertices. Interestingly, in spherical three-space, flexible octahedra with six vertices can be constructed, see the case $n = 3$ of Theorem 1.1 in \cite{cross_polytopes}. The focus of our discussion will be on Euclidean three-space, and when we say a polyhedron is embedded, we mean that it is embedded in $\R^3$.

Inspired by the paper \cite{Minimal_example} and the work of G. Nelson in \cite{Nelson_pentagons}, we introduce a new method for constructing flexible polyhedra. The method of cutting along a symmetric quadrilateral, followed by a twist or reflection, is suitable for making new flexible polyhedral surfaces from old ones.

The minimal polyhedron found in \cite{Minimal_example} is an example of our method, as it arises by cutting and twisting a Bricard type I octahedron, as discussed in section \ref{S2}. This explanation of the construction simplifies their proof of flexibility. By following their strategy, but instead cutting and reflecting along a quadrilateral on a flexible octahedron, we can derive an embedded flexible dodecahedron with a significantly larger range of motion. In section \ref{S2} we give an example of parameters for a good working model without self-intersections. A net can be seen in Figure \ref{net_dodecahedron}, and a model at \url{https://www.geogebra.org/m/pb4nqczx}.

In section \ref{S1} of this paper, we prove the geometric lemmata about symmetric quadrilaterals required for our construction. Section \ref{S2} explains the method of cutting along a symmetric quadrilateral, followed by a specific example of cutting and reflecting to find a flexible dodecahedron. Figure \ref{pentagonal_bipyramids} explains how the construction works. In section \ref{S3}, we discuss the steps required to see that eight vertices are minimal.

\section{Symmetric quadrilaterals in $\R^3$}
\label{S1}

The paper \cite{twinning_paper} shows how to construct an infinite family of flexible polyhedra using a method called twinning. The basis of the construction is the following two theorems, explaining the motion of symmetric quadrilaterals in $\R^3$. We will make use of these theorems to create new flexible polyhedra from old by cutting them along symmetric quadrilaterals.

Theorem \ref{line_symmetry} is used in describing the Bricard type I octahedron and type I twinning. Theorem \ref{plane_symmetry} desribes the Bricard type II octahedron and type II twinning. These theorems are well known and have been used before to explain the flexible Bricard octahedra in \cite{Kuiper} and \cite{p1}. For completeness, we include the proofs below.

A quadrilateral flexing in three dimensions has $3\cdot 4 - 4 = 8$ degrees of freedom. Disregarding the six degrees of freedom coming from Euclidean isometries, that leaves two degrees of freedom. So we can regard the space of possible positions as a two dimensional surface. Theorem \ref{line_symmetry} states that a rotationally symmetric quadrilateral will remain rotationally symmetric throughout any flexing motion, and similarly Theorem \ref{plane_symmetry} in the case of a quadrilateral with a plane symmetry.

\begin{theorem}
\label{line_symmetry}
Let $A$, $B$, $A'$, $B'$ be four points in $\mathbb{R}^3$ such that $AB = A'B'$ and $AB' = A'B$. Then there exists a line $l$ in $\mathbb{R}^3$ such that a half-rotation in $l$ swaps $A$ with $A'$, and $B$ with $B'$.
\end{theorem}
\begin{proof} If the diagonals $AA'$ and $BB'$ coincide, the four points lie on a single line, in which case the symmetry is evident from the construction. In the case when the diagonals $AA'$ and $BB'$ intersect in a point $M$, the quadrilateral $ABA'B'$ is planar, so it is a parallelogram. Therefore, it has a rotational symmetry in a line through $M$ perpendicular to a plane containing the four points.

Assuming now that the diagonals do not intersect, let $X$ be the midpoint of $AA'$ and $Y$ be the midpoint of $BB'$, as shown in Figure \ref{theorem1_fig}. By assumption, these are different points, and let $l$ be the line through $X$ and $Y$.

\begin{figure}\centering
\begin{tikzpicture}[line join=bevel,scale = 1.5,tick/.style={
    postaction={decorate},
    decoration={
      markings,
      mark=at position 0.5 with {
        \draw[-] (-2pt,2pt) -- (2pt,-2pt);
      }
    }
  }, doubletick/.style={
    postaction={decorate},
    decoration={
      markings,
      mark=at position 0.485 with {
        \draw[-] (-2pt,2pt) -- (2pt,-2pt);
      },  mark=at position 0.515 with {
        \draw[-] (-2pt,2pt) -- (2pt,-2pt);
      }
    }
  }]
\coordinate (B) at (1,1.5,1.5);
\coordinate (A) at (2,-1,-1);
\coordinate(B') at (-1,-1.5,1.5);
\coordinate (A') at (-2,1,-1);
\coordinate(Y) at (0,0,1.5);
\coordinate(X) at (0,0,-1);
\coordinate(X') at (0,0,2.7);
\coordinate(Y') at (0,0,-2.7);
\draw[tick](A)--(B);
\draw[doubletick](A')--(B);
\draw[tick](A')--(B');
\draw[doubletick](B')--(A);
\draw[dashed](B)--(B');
\draw(A)--(A');
\draw[dotted](A')--(Y);
\draw[dotted](A)--(Y);
\draw[fill=black] (A) circle [radius=0.7pt] node[below] {$A$};
\draw[fill=black] (B) circle [radius=0.7pt] node[above right] {$B$};
\draw[fill=black] (A') circle [radius=0.7pt] node[above left] {$A'$};
\draw[fill=black] (B') circle [radius=0.7pt] node[above left] {$B'$};
\draw[fill=black] (X) circle [radius=0.7pt] node[below] {$X$};
\draw[fill=black] (Y) circle [radius=0.7pt] node[below right] {$Y$};
\draw[dashed] (X)--(X');
\draw(X)--(Y');
\draw (Y') node[ below right]{$l$};
\end{tikzpicture}
\caption{The quadrilateral $ABA'B'$ has rotational symmetry.}
\label{theorem1_fig}
\end{figure}

Since $\triangle{ABB'} \cong \triangle{BA'B'}$, and $Y$ is the midpoint of $BB'$, we have $$AY = A'Y,$$ since they are corresponding medians in congruent triangles.  Therefore, we have $\triangle{AXY} \cong \triangle{A'XY}$, since the corresponding sides in these triangles are equal. In particular, $$\angle{AXY} = \angle{A'XY},$$ and since $A',X,A$ lie on a line, we have $XY \perp AA'.$

Similarly, we obtain $XY \perp BB',$ and so it follows that $l$ is the line of symmetry of the quadrilateral $ABA'B'$.
\end{proof}

The second theorem shows how reflectional symmetry is maintained in a flexing quadrilateral.

\begin{theorem}
\label{plane_symmetry}
Let $A$, $B$, $A'$, $B'$ be four points in $\mathbb{R}^3$ such that $AB = AB'$ and $A'B = A'B'$. Then there exists a plane $\pi$ in $\mathbb{R}^3$ through $A$ and $A'$, such that reflecting in $\pi$ swaps $B$ and $B'$.
\end{theorem}

\begin{proof}
If $B$ and $B'$ coincide, then the plane $\pi$ can be taken as a plane containing the points $A,B,A'$. If the four points are distinct and coplanar, the quadrilateral $ABA'B'$ is a kite, and so a plane containing $A$ and $A'$, perpendicular to a plane of the quadrilateral will work. Note also that when the four points lie on a line, $B$ and $B'$ coincide.
\begin{figure}\centering
\begin{tikzpicture}[line join=bevel,scale = 1.6,tick/.style={
    postaction={decorate},
    decoration={
      markings,
      mark=at position 0.49 with {
        \draw[-] (-2.1pt,2.1pt) -- (-2.1pt,-2.1pt);
      }
    }
  }, doubletick/.style={
    postaction={decorate},
    decoration={
      markings,
      mark=at position 0.505 with {
        \draw[-] (-2.1pt,2.1pt) -- (-2.1pt,-2.1pt);
      },  mark=at position 0.54 with {
        \draw[-] (-2.1pt,2.1pt) -- (-2.1pt,-2.1pt);
      }
    }
  }]
\coordinate (B) at (1.8,-1.5,1.5);
\coordinate (A) at (2,0,0);
\coordinate(B') at (1.8,-1.5,-1.5);
\coordinate (A') at (-2,0,0);
\coordinate(M) at (1.8,-1.5,0);
\coordinate(X) at (3,0.8,0);
\coordinate(Y) at (-3,0.8,0);
\coordinate(X') at (-3,-2.6,0);
\coordinate(Y') at (3,-2.6,0);
\draw[dashed](B')--(M);
\draw(A)--(A');
\draw[dashed,tick](A')--(B');
\draw[dashed, doubletick](A)--(B');
\draw(M)--(B);
\draw[tick](A')--(B);
\draw[doubletick](A)--(B);
\draw(A')--(M);
\draw(A)--(M);
\draw[fill opacity = 0.05, fill = black, dotted] (X)--(Y)--(X')--(Y')--cycle;
\draw[fill=black] (A) circle [radius=0.5pt] node[above right] {$A$};
\draw[fill=black] (B) circle [radius=0.5pt] node[below right] {$B$};
\draw[fill=black] (A') circle [radius=0.5pt] node[above left] {$A'$};
\draw[fill=black] (B') circle [radius=0.5pt] node[below right] {$B'$};
\draw[fill=black] (M) circle [radius=0.5pt] node[below right] {$M$};
\end{tikzpicture}
\caption{The quadrilateral $ABA'B'$ has reflective symmetry. The point $B'$ lies behind the grey plane $\pi$.
}
\label{theorem2_fig}
\end{figure}

In the other case, refer to Figure \ref{theorem2_fig}.  Let $M$ be the midpoint of $BB'$. Since the quadrilateral $ABA'B'$ is non-planar, the points $A$,$A'$, and $M$ do not lie on one line, so we let $\pi$ be the plane through $A$, $A'$, and $M$. The triangle $\triangle{BA'B'}$ is isosceles, so $$\angle{A'MB} = \angle{B'MA'} = 90^{\circ}.$$ Since $\triangle{BAB'}$ is isosceles, we also have $$\angle{AMB} = \angle{B'MA} = 90^{\circ}.$$ 

Since we have $BB' \perp AM$ and $BB' \perp A'M$, and $A$,$A'$, and $M$ are not collinear, we get $BB' \perp \pi$. The symmetry in the plane $\pi$ follows as $M$ is the midpoint of $BB'$.
\end{proof}

In the next section, we will see how these two theorems can be used to cut up and glue flexible polyhedral surfaces to obtain new models with the same freedom to flex.

\section{Cutting along a quadrilateral}
\label{S2}

\subsection{General description}

 We describe a generalised version of a method outlined by G. Nelson in \cite{Nelson_pentagons}.\footnote{Gerald Nelson was a retired software engineer from Minnesota who took an interest in flexible polyhedra late in life. His contributions are on arXiv, and are worth a read.} His paper is the main inspiration used in \cite{Minimal_example} to create a new minimal example of an embedded flexible polyhedron. We call the method \textit{cutting along a quadrilateral, followed by a twist or reflection}. The following description is simplified and more general than the constructions that have previously appeared. In particular, the notion of arriving at these models by cutting up a known model is new. Also, Nelson only discusses the specific case starting from an octahedron, and does not consider the possibility of cutting and reflecting, which is what we will use to construct an embedded dodecahedron with a large range of motion in the next subsection.

Given a flexible polyhedron $P$, let $A,B,A',B'$ be some points on the edges of $P$, such that the edges $AB$, $BA'$, $A'B'$ and $B'A$ all lie on the faces of $P$, and such that $AB = A'B'$ and $A'B = AB'$. We can then cut $P$ along the quadrilateral $ABA'B'$ to obtain two caps on a quadrilateral hole, and these can flex individually in the same way as when $P$ flexed. The shared quadrilateral base $ABA'B'$ is such that opposite edges are of the same length, and so we can apply Theorem \ref{line_symmetry} to find its line of symmetry, $l$. We turn one of the pieces by $180^{\circ}$ in the line $l$, and glue the two pieces back together along $ABA'B'$, to obtain a new flexible polyhedron $P^{\text{I}}_{ABA'B'}$. We call this the \textit{cut and twisted polyhedron of $P$, along the quadrilateral $ABA'B'$}. It will be flexible, since $P$ is flexible.

A \textit{cut and reflected polyhedron} can be defined in an analogous manner, for  $A,B,A',B'$ some points on the edges of $P$, such that the edges $AB$, $BA'$, $A'B'$ and $B'A$ all lie on the faces of $P$, and such that $AB = AB'$ and $AB' = A'B'$. The quadrilateral $ABA'B'$ has a plane of symmetry $\pi$ by Theorem \ref{plane_symmetry}. We cut off a piece along $ABA'B'$, reflect it in $\pi$, and glue the two boundaries back together. In this way, we construct a new polyhedron $P^{\text{II}}_{ABA'B'}$. This polyhedron has the same freedom to flex as $P$. See Figure \ref{pentagonal_bipyramids} for a diagram depicting this when $P$ is a Bricard octahedron.

Let us summarise this in a theorem. Note that $P$ can be any polyhedral surface, so long as the quadrilateral $ABA'B'$ splits it into two pieces.

\begin{theorem}
\label{cutting_thm}
Let $P$ be a triangulated surface, with $A,B,A',B'$ points on the edges of $P$, such that the quadrilateral $ABA'B'$ lies on the surface of $P$, and splits it into two parts. If $ABA'B'$ has a rotational (respectively reflectional) symmetry, then the surface $P^{\text{I}}_{ABA'B'}$ (respectively $P^{\text{II}}_{ABA'B'}$) has the same freedom to flex as $P$.
\end{theorem}

We will show below how such a cut and reflection, applied to a flexible octahedron, can be used to create flexible pentagonal bipyramids. By a suitable choice of parameters, and the addition of a tent on one of the faces, this leads to a new minimal example of a flexible polyhedron, a dodecahedron on 8 vertices.

\subsection{A flexible dodecahedron}

\begin{figure}
\centering
\subfloat{

\begin{tikzpicture}[line join=bevel,z=-5.5,scale = 0.56\textwidth/3.6cm]
\coordinate (A1) at (0,0,-1);
\coordinate (A2) at (-1,0.5,0);
\coordinate (A3) at (0,0,1);
\coordinate (A4) at (1,0.5,0);
\coordinate (B1) at (-1,-2.5,-0.4);
\coordinate (C1) at (1,-2.5,0.4);
\coordinate (M1) at ($(A2)!0.45!(C1)$);
\coordinate (E) at ($(A2)!0.48!(C1)$);

\draw [fill opacity=0.7,fill=white] (A1) -- (A2) -- (B1) -- cycle;
\draw [fill opacity=0.7,fill=white](A4) -- (A1) -- (C1) -- cycle;
\draw [fill opacity=0.7,fill=white] (A4) -- (A1) -- (B1) -- cycle;
\draw [fill opacity=0.7,fill=white] (A1) -- (A2) -- (C1) -- cycle;
\draw (M1)--(A2);
\draw [red] (A1)--(E);
\draw  [fill opacity=0.7,fill=white] (A3) -- (A4) -- (B1) -- cycle;
\draw  [fill opacity=0.7,fill=white, draw = none] (A2) -- (A3) -- (C1) -- cycle;
\draw  [fill opacity=0.7,fill=white] (A2) -- (A3) -- (B1) -- cycle;
\draw  [fill opacity=0.7,fill=white](A3) -- (A4) -- (C1) -- cycle;
\draw (M1) -- (C1);
\draw [blue] (A3)-- (E);
\draw [red, thin,dotted] (A1)-- (E);
\draw[blue] (A1)--(A2);
\draw[red] (A3)--(A2);
\draw[blue] (A3)--(A4);
\draw[red] (A4)--(A1);
\node[below left] at (E) {$D$};
\node[above] at (A1) {$B'$};
\node[above] at (A2) {$A$};
\node[left] at (A3) {$B$};
\node[above] at (A4) {$A'$};
\node[left] at (C1) {$C'$};
\node[below] at (B1) {$C$};
\end{tikzpicture}}
\qquad
\subfloat{
\begin{tikzpicture}[line join=bevel,z=-5.5,scale = 0.56\textwidth/3.6cm]
\coordinate (A1) at (0,0,-1);
\coordinate (A2) at (-1,0.5,0);
\coordinate (A3) at (0,0,1);
\coordinate (A4) at (1,0.5,0);
\coordinate (B1) at (-1,-2.5,-0.4);
\coordinate (C1) at (1,-2.5,0.4);
\coordinate (C1') at (2.5,0,0);
\coordinate (M1) at ($(A2)!0.45!(C1)$);
\coordinate (M3) at ($(A1)!0.17!(C1')$);
\coordinate (E) at ($(A2)!0.48!(C1)$);
\coordinate (M4) at ($(A1)!0.17!(E)$);

\draw [fill opacity=0.7,fill=white] (A1) -- (A2) -- (B1) -- cycle;
\draw [fill opacity=0.7,fill=white] (A4) -- (A1) -- (B1) -- cycle;
\draw  [fill opacity=0.7,fill=white, draw = none](A1) -- (E) -- (C1') -- cycle;

\draw [fill opacity=0.7,fill=white, draw = none] (A1) -- (A2) -- (E) -- cycle;
\draw [red] (A1)--(E);

\draw [fill opacity=0.7,fill=white](A4) -- (A1) -- (C1') -- cycle;
\draw (M1)--(A2);
\draw  [fill opacity=0.7,fill=white] (A3) -- (A4) -- (B1) -- cycle;
\draw  [fill opacity=0.7,fill=white, draw = none] (A2) -- (A3) -- (E) -- cycle;
\draw  [fill opacity=0.7,fill=white] (A2) -- (A3) -- (B1) -- cycle;
\draw  [fill opacity=0.7,fill=white](A3) -- (A4) -- (C1') -- cycle;
\draw  [fill opacity=0.7,fill=white](A3) -- (E) -- (C1') -- cycle;
\draw[opacity = 0.5,gray] (M3)--(C1');
\draw (M1) -- (E);
\draw [blue] (A3)-- (E);
\draw[blue] (A1)--(A2);
\draw[red] (A3)--(A2);
\draw[blue] (A3)--(A4);
\draw[red] (A4)--(A1);
\draw [red,dotted] (M4)-- (E);
\node[below left] at (E) {$D$};
\node[above] at (A1) {$B'$};
\node[above] at (A2) {$A$};
\node[left] at (A3) {$B$};
\node[above] at (A4) {$A'$};
\node[below] at (C1') {$C''$};
\node[below] at (B1) {$C$};
\end{tikzpicture}}
\caption{Cutting and reflecting a Bricard type I octahedron along the quadrilateral $DB'A'B$. The red and blue colours indicate segments of equal lengths, $AB' = A'B = BD$ and $AB = B'A' = B'D$.}
\label{pentagonal_bipyramids}
\end{figure}

We proceed to show by example a simple case of cutting along a quadrilateral and reflecting. From a flexible octahedron containing a symmetric quadrilateral, we obtain a flexible bipyramid. With a suitable choice of edge lengths, and the addition of a tent, this can be turned into an embedded flexible dodecahedron.

Start with a special type of Bricard type I octahedron, depicted in the left part of Figure \ref{pentagonal_bipyramids}, constructed as follows. Begin with four vertices $A,B,A',B'$ in $\R^3$, not all on one line, such that $AB = A'B'$, and $A'B = AB'$. Next, since $A,B,A',B'$ are not collinear, choose a point $D$ different from $A'$ with $DB = BA'$ and $DB' = B'A'$. Choose a point $C'$ on the ray $AD$ to be one of the vertices of the Bricard octahedron. The final vertex, $C$, is obtained from $C'$ by a half-turn in the line of symmetry of the quadrilateral $ABA'B'$, given by Theorem \ref{line_symmetry}. Denote by $P$ the octahedron with vertices $A,B,A',B',C,$ and $C'$, where opposite vertices are a dashed and undashed pair. This is a Bricard type I octahedron, whose flexibility follows from the rotational symmetry of $ABA'B'$ given by Theorem \ref{line_symmetry}. Furthermore, since the point $D$ on the edge $AC'$ satisfies $DB = BA'$ and $DB' = B'A'$, the quadrilateral $DB'A'B$ has a plane of symmetry by Theorem \ref{plane_symmetry}. As indicated in the right part of Figure \ref{pentagonal_bipyramids}, we cut the polyhedron $P$ along $DB'A'B$, and reflect the point $C'$ in the plane of symmetry of $DB'A'B$ into $C''$. We remove the point $C'$ and instead draw the edges $DC''$, $BC''$, $A'C''$, $B'C''$, as seen on the right part of Figure \ref{pentagonal_bipyramids}. This new polyhedron is $P_{BDB'A'}^{II}$, and by Theorem \ref{cutting_thm}, this is flexible just like the octahedron $P$. We have added a single vertex to $P$ at the point $D$, and so the result is a polyhedron on seven vertices. The polyhedron $P_{BDB'A'}^{II}$ has ten faces, and so is a decahedron. Its combinatorial structure can be recognised as that of a pentagonal bipyramid, where the two cones are erected on the pentagon $ACA'C''D$, with apices at $B$ and $B'$.

The polyhedron $P_{BDB'A'}^{II}$ still has self-intersection. By drawing a model in GeoGebra and varying the edge lengths, the author quickly found parameters such that all of the self-intersections involve one of the faces. These self-intersections can then be eliminated by erecting a tent onto that face. By adding a tent, we add a single vertex and two faces, and so the resulting polyhedron is a dodecahedron on eight vertices. Thus, we have arrived at a new example of an embedded flexible polyhedron on eight vertices. As discussed in section \ref{S3}, eight is the smallest possible number of vertices.

A net for a particular realisation of this polyhedron, attaining a large range of motion, is shown in Figure \ref{net_dodecahedron}. The values $l_1,l_2,l_3,l_4,l_5$ were chosen by eye using sliders in GeoGebra, thus determining $x$ and $y$. Here, we have set 
\begin{equation*}
l_1 = 3.6 \text{, } l_2 = 3.9 \text{, } l_3 = 1\text{, } l_4 = 3.9 \text{, and }l_5 = 2.9.\end{equation*}
By an application of the law of cosines, we can compute $x$ and $y$ as \begin{align*} x =& \sqrt{l_2^2+(l_3+l_4)^2 - \dfrac{l_3+l_4}{l_3}(l_2^2+l_3^2-l_1^2)} =\dfrac{\sqrt{9318}}{20} \approx 4.83,\\
\text{and }y =& \sqrt{l_1^2+(l_3+l_4)^2 - \dfrac{l_3+l_4}{l_3}(l_1^2+l_3^2-l_2^2)} = \dfrac{13 \sqrt{102}}{20} \approx 6.56.\end{align*} The altitudes in the tetrahedral tent are chosen as $h_1 = 6.5$, $h_2 = 6.5$, $h_3 = 6.1$.

\begin{figure}
    \centering
\definecolor{ffdxqq}{rgb}{1,0.8431372549019608,0}
\definecolor{ttffcc}{rgb}{0.2,1,0.8}
\definecolor{ffqqqq}{rgb}{1,0,0}
\definecolor{fuqqzz}{rgb}{0.9568627450980393,0,0.6}
\definecolor{qqqqff}{rgb}{0,0,1}
\definecolor{xfqqff}{rgb}{0.4980392156862745,0,1}
\definecolor{qqwuqq}{rgb}{0,0.39215686274509803,0}
\begin{tikzpicture}[line cap=round,line join=round,>=triangle 45,scale = \textwidth/20cm, every node/.style={scale=1.4}]
\clip(-8.8,-3) rectangle (11.5,6.5);
\draw [line width=1pt] (0,0)-- (3.2823391049676753,5.610369863030423);
\draw [line width=1pt] (3.2823391049676753,5.610369863030423)-- (6.564678209935351,0);
\draw [line width=1pt,dash pattern=on 4pt off 4=pt] (6.564678209935351,0)-- (0,0);
\draw [line width=1pt,color=qqwuqq] (0,0)-- (3.8002624351400898,-0.876359186671261);
\draw [line width=1pt,color=xfqqff] (3.8002624351400898,-0.876359186671261)-- (6.564678209935351,0);
\draw [line width=1pt] (0,0)-- (-2.16906359992303,2.873179962948535);
\draw [line width=1pt,color=qqqqff] (-2.16906359992303,2.873179962948535)-- (3.2823391049676753,5.610369863030423);
\draw [line width=1pt,color=qqqqff] (3.2823391049676753,5.610369863030423)-- (9.200198568453429,4.130863401256056);
\draw [line width=1pt] (9.200198568453429,4.130863401256056)-- (6.564678209935351,0);
\draw [line width=1pt,dash pattern=on 4pt off 4pt] (-4.8264533329926875,0.018660772861353703)-- (-2.16906359992303,2.873179962948535);
\draw [line width=1pt,dash pattern=on 4pt off 4pt,color=fuqqzz] (0,0)-- (-4.8264533329926875,0.018660772861353703);
\draw [line width=1pt,color=ffqqqq] (-2.16906359992303,2.873179962948535)-- (-3.656745217007687,6.478287923355321);
\draw [line width=1pt,dash pattern=on 4pt off 4pt] (-3.656745217007687,6.478287923355321)-- (-4.8264533329926875,0.018660772861353703);
\draw [line width=1pt,dash pattern=on 4pt off 4pt] (-4.8264533329926875,0.018660772861353703)-- (-5.196219214951344,2.894990585754746);
\draw [line width=1pt,color=ttffcc] (-5.196219214951344,2.894990585754746)-- (-3.656745217007687,6.478287923355321);
\draw [line width=1pt] (-8.793947305432228,2.7671133517966213)-- (-5.196219214951344,2.894990585754746);
\draw [line width=1pt,dash pattern=on 4pt off 4pt,color=fuqqzz] (-8.793947305432228,2.7671133517966213)-- (-4.8264533329926875,0.018660772861353703);
\draw [line width=1pt,color=ffdxqq] (-5.057699702691361,3.885350290469784)-- (-5.196219214951344,2.894990585754746);
\draw [line width=1pt,color=qqwuqq] (-5.057699702691361,3.885350290469784)-- (-8.793947305432228,2.7671133517966213);
\draw [line width=1pt] (6.564678209935351,0)-- (11.315006588227709,0.8540376445978861);
\draw [line width=1pt,color=ffqqqq] (11.315006588227709,0.8540376445978861)-- (9.200198568453429,4.130863401256056);
\draw [line width=1pt,color=xfqqff] (6.564678209935351,0)-- (8.857042787448917,-1.7762501636202728);
\draw [line width=1pt] (8.857042787448917,-1.7762501636202728)-- (11.315006588227709,0.8540376445978861);
\draw [line width=1pt,color=ttffcc] (11.315006588227709,0.8540376445978861)-- (9.458047006586408,-2.575496013518123);
\draw [line width=1pt,color=ffdxqq] (9.458047006586408,-2.575496013518123)-- (8.857042787448917,-1.7762501636202728);
\begin{scriptsize}
\draw[color=black] (2,2.8) node {$h_1$};
\draw[color=black] (5.3,2.8) node {$h_2$};
\draw[color=black] (3.4,0.3) node {y};
\draw[color=black] (1.9,-0.7) node {$l_2$};
\draw[color=black] (5.5,-0.7) node {$l_5$};
\draw[color=black] (-0.8,1.7) node {$l_1$};
\draw[color=black] (0.3,4.5) node {$h_3$};
\draw[color=black] (6.4,5.15) node {$h_3$};
\draw[color=black] (8.7,2) node {$l_4+l_3$};
\draw[color=black] (-3,1.5) node {$l_4$};
\draw[color=black] (-2.4,-0.2) node {x};
\draw[color=black] (-2.7,5) node {$l_2$};
\draw[color=black] (-4,3.1) node {y};
\draw[color=black] (-5.3,1.6) node {$l_5$};
\draw[color=black] (-4.6,5.1) node {$l_2$};
\draw[color=black] (-6.7,2.5) node {$l_1$};
\draw[color=black] (-6.9,1.1) node {x};
\draw[color=black] (-5.4,3.4) node {$l_3$};
\draw[color=black] (-7,3.7) node {$l_2$};
\draw[color=black] (9,0.2) node {x};
\draw[color=black] (10.5,2.8) node {$l_2$};
\draw[color=black] (7.6,-1.1) node {$l_5$};
\draw[color=black] (9.6,-0.4) node {$l_1$};
\draw[color=black] (10.7,-1) node {$l_2$};
\draw[color=black] (8.95,-2.25) node {$l_3$};
\end{scriptsize}
\end{tikzpicture}
    \caption{A net for a flexible dodecahedron. Valley folds are given by dashed lines and mountain folds by solid lines. The dotted edge marked $y$ can go from a valley fold to a mountain fold during the flexing motion, so should be scored on both sides. Gluing instructions are indicated by the colour of the edges. A virtual model is available at \url{https://www.geogebra.org/m/pb4nqczx}. }
    \label{net_dodecahedron}
\end{figure}

Note two benefits of the cut and reflection construction above compared with the method discussed in \cite{Minimal_example}. First, finding an example with the desired self-intersection was easier, and second, the final model reported here has a larger range of movement. These properties are related. It seems that in the present case, the cut and reflection is slightly better at shifting the self-intersection in the correct way. The previous work, arising from the cut and twist, involved a large computer search, whereas this search was done by human trial and error over an afternoon.\footnote{A computational optimisation, like that of \cite{Guest_optimisation}, led to results similar to those reported there for Steffen's Polyhedron. It is possible to roughly double the range of motion for the model, but this comes at the cost of some triangles being nearly degenerate. When taking this degeneration into account, a randomised optimisation found some reasonable polyhedra. One polyhedron found in this way, yielding an approximately $25\%$ larger range of motion, comes from setting $l_1 = 4.2,$ $l_2 = 4.3,$ $l_3 = 1,$ $l_4 = 4.8,$ and $l_3 = 3.05$. Then one can set e.g. $h_1 = 7.9, h_2 = 4, h_3 = 6.4$. However, the parameters used above make for a model that is more aesthetically pleasing and easier to assemble from paper.}

The polyhedron described in \cite{Minimal_example} can also be explained in an analogous manner, from cutting and twisting a Bricard type I octahedron along a quadrilateral. Refer to figure \ref{bipyramid_figure}, where vertices are labelled as in the original paper. We start with a Bricard type I octahedron on the base $p_1p_Tp_3p_B$, where $p_1p_T = p_3p_B$, and $p_Tp_3 = p_Bp_1$. The apices are at $p_0$ and $p_2$, where $p_0$ is the rotation of $p_2$ in the line of symmetry of the quadrilateral $p_1p_Tp_3p_B$, as described by Theorem \ref{line_symmetry}. Similarly to the previous case, the edge lengths are chosen so that there is a point $p_5$ on the extension of the segment $p_0p_1$, such that $p_0$ lies between $p_5$ and $p_1$, $p_5p_T = p_1p_T$, and $p_5p_B = p_1p_B$. We next cut and twist along the rotationally symmetric quadrilateral $p_5p_Tp_3p_B$. Since $p_5$ lies on the extension of $p_1p_0$, we need to interpret theorem \ref{cutting_thm} accordingly. When we cut off the two triangles that are not present, we extend the faces $p_Tp_1p_0$ and $p_Bp_1p_0$ to include the two triangles $p_0p_5p_T$ and $p_0p_5p_B$.  When cutting and twisting, the vertex originally at $p_0$ gets sent to $p_4$. Also note how the edge $p_0p_3$, present at the start, becomes the edge $p_5p_4$. As a side note, this is in fact how the cut and twist helps us to get rid of the self-intersections, since before we cut, the edge $p_0p_3$ is on the wrong side of the face $p_2p_3p_B$. This argument shows how Theorem \ref{cutting_thm} offers an alternative proof of flexibility for the pentagonal bipyramid in \cite{Minimal_example}. A flexing model can be seen at \url{https://www.geogebra.org/m/jhsxhjzx}, where congruent faces are marked with the same colour.

\begin{figure}
\centering
\tdplotsetmaincoords{-60}{130}
\subfloat{
\begin{tikzpicture}[line join=bevel, scale = 1.3]

\begin{scope}[tdplot_main_coords]
\coordinate (p0) at (2.52,2.55,2.22);
\coordinate (p1) at (0,0,0);
\coordinate (p2) at (7.52,0.42,0.71);
\coordinate (p3) at (6.69,-0.41,-3.35);
\coordinate (p4) at (5.82,-1.3,-2.2);
\coordinate (p5) at (3.53,3.57,3.11);
\coordinate (pb) at (3.54,1.39,0);
\coordinate (pt) at (4.94,0,0);

\draw[fill opacity=0.7,fill=white] (p0)--(pb)--(p1)--cycle;
\draw[fill opacity=0.7,fill=white] (p0)--(pt)--(p3)--cycle;
\draw[fill opacity=0.7,fill=white] (p0)--(p3)--(pb)--cycle;
\draw[fill opacity=0.7,fill=white] (p2)--(p3)--(pb)--cycle;
\draw[red] (p3)--(pb);
\draw (p2)--(pb)--(p1)--cycle;
\draw [blue] (pb)--(p1);

\draw[dashed] (p0)--(p5);

\draw [blue,dashed] (p5)--(pb);

\node[right] at (pb) {$p_B$};
\draw (p0)--(pt)--(p3);
\draw[fill opacity=0.7,fill=white] (p0)--(p1)--(pt)--cycle;
\draw[red,dashed] (p5)--(pt);
\draw[fill opacity=0.7,fill=white] (p2)--(p1)--(pt)--cycle;
\draw[red] (p1)--(pt);
\draw[fill opacity=0.7,fill=white] (p2)--(pt)--(p3)--cycle;
\draw [blue] (pt)--(p3);

\node[right] at (p0) {$p_0$};
\node[above] at (p1) {$p_1$};
\node[below] at (p2) {$p_2$};
\node[above] at (p3) {$p_3$};
\node[below] at (p5) {$p_5$};

\node[left] at (pt) {$p_T$};

\end{scope}

\end{tikzpicture}}
\quad
\subfloat{
\begin{tikzpicture}[line join=bevel, scale = 1.3]
\begin{scope}[tdplot_main_coords]
\coordinate (p0) at (2.52,2.55,2.22);
\coordinate (p1) at (0,0,0);
\coordinate (p2) at (7.52,0.42,0.71);
\coordinate (p3) at (6.69,-0.41,-3.35);
\coordinate (p4) at (5.82,-1.3,-2.2);
\coordinate (p5) at (3.53,3.57,3.11);
\coordinate (pb) at (3.54,1.39,0);
\coordinate (pt) at (4.94,0,0);

\draw[fill opacity=0.7,fill=white] (p5)--(pb)--(p1)--cycle;

\draw[fill opacity=0.7,fill=white] (p2)--(p3)--(pb)--cycle;
\draw[fill opacity=0.7,fill=white] (p3)--(pt)--(p4)--cycle;

\draw (p2)--(pb)--(p1)--cycle;
\draw [blue] (pb)--(p1);

\draw[fill opacity=0.7,fill=white] (p2)--(pt)--(p3)--cycle;
\draw[fill opacity=0.7,fill=white] (p3)--(pt)--(p4)--cycle;
\draw [blue] (pt)--(p3);
\draw[fill opacity=0.7,fill=white] (p5)--(pb)--(p4)--cycle;
\draw [blue] (p5)--(pb);
\draw[fill opacity=0.7,fill=white] (p5)--(pt)--(p4)--cycle;
\draw[fill opacity=0.7,fill=white] (p3)--(pb)--(p4)--cycle;
\draw[red] (p3)--(pb);
\node[right] at (pb) {$p_B$};
\draw[fill opacity=0.7,fill=white] (p5)--(p1)--(pt)--cycle;
\draw[red] (p5)--(pt);
\draw[red] (p1)--(pt);
\draw[fill opacity=0.7,fill=white] (p2)--(p1)--(pt)--cycle;

\node[above] at (p1) {$p_1$};
\node[below] at (p2) {$p_2$};
\node[above] at (p3) {$p_3$};
\node[left] at (p4) {$p_4$};
\node[below] at (p5) {$p_5$};

\node[left] at (pt) {$p_T$};

\end{scope}

\end{tikzpicture}}
\caption{Recovering the flexible pentagonal bipyramid from \cite{Minimal_example} using a cut and twist along the symmetric quadrilateral $p_5p_Tp_3p_B$ in the Bricard octahedron on the left. The face of self-intersection, $p_2p_1p_B$, is omitted.}
\label{bipyramid_figure}
\end{figure}

Notice that the range of motion of the polyhedron from \cite{Minimal_example} is smaller than for the dodecahedron in figure \ref{net_dodecahedron}. Experimentally, it seems more difficult to find working parameters for the cut and twisted version, and those found have not led to a large range of motion. This is why we focused on the cut and reflection above.

The construction we described in Figure \ref{pentagonal_bipyramids} can also be realised in the language of the paper \cite{Minimal_example}. It arises from gluing together a Bricard type I octahedron and a Bricard type II octahedron. During the gluing, the two polyhedra share the cone $BDA'B'C'$, which is then removed.

\section{Minimality}
\label{S3}

\begin{figure}
    \centering
   
\tikz\node[draw, inner sep=3mm,scale = 1.15] at (0,0) {
\tikz[circle, minimum size=2.5mm,inner sep=0pt]
\graph [nodes={fill=purple}, empty nodes] { a -- {c[y = -1],d[y = -1]} -- e,c--d,a--e };{\small$V = 4$}};
\tikz\node[draw, inner sep=3mm,scale = 1.15] at (0,0) {
\tikz[circle, minimum size=2.5mm,inner sep=0pt]
\graph [nodes={fill=blue!70}, empty nodes] { a -- {b[x = 0, y = -0.4, purple],c,d[purple]} -- e, b--c,c--d,a--e };{\small$V = 5$}};
\tikz\node[draw, inner sep=3mm, scale = 0.95] at (0,0) {
\tikz[circle, minimum size=2.5mm,inner sep=0pt]
\graph [nodes={fill=blue!70}, empty nodes]{a[x = -0.3] --{x1[y = 1.2],x2[x = -0.2, y = 1.3],x3[x=0.2,y = 1.6],x4[y = 1.6]} -- b[x = 0.3],x1--x2--x4--x3--x1};
\tikz[circle, minimum size=2.5mm,inner sep=0pt]
\graph [nodes={fill=blue!70}, empty nodes] { a -- {b[x = 0, y = -0.4, purple],c[y = 0.1],d[y = -0.3]} -- e,c--d,a--e, x1[x = 0.6, y = 2.1,purple]--{a,c,d}, b--c};{$V = 6$}};
\vspace{1mm}

\tikz\node[draw, inner sep=3mm, scale = 0.9] at (0,0) {
\tikz[circle, minimum size=2.5mm,inner sep=0pt]
\graph [nodes={fill=blue!70}, empty nodes] {b[x=0.5, y = 0.5]-- subgraph C_n [n=5, clockwise,radius =45] --a[x = 0.5,y = -0.2]};
\tikz[circle, minimum size=2.5mm,inner sep=0pt]
\graph [nodes={fill=blue!70}, empty nodes]{a[x = -0.3] --{x1[y = 1.2],x2[x = -0.2, y = 1.3],x3[x=0.2,y = 1.6],x4[y = 1.6]} -- b[x = 0.3],x1--x2--x4--x3--x1, {x4,x3,b} -- y[y = 3.3,x = 1.35,purple]};
\tikz[circle, minimum size=2.5mm,inner sep=0pt]
\graph [nodes={fill=blue!70}, empty nodes] { a[y = 0.5,x = -0.5] -- {b[x = 0, y = -0.4],c[y = 0.1],d[y = -0.3]} -- e[y = 0.5, x= 0.5],c--d,a--e, x1[x = 0.6, y = 2.1,purple]--{a,c,d}, b--c, x3[y = 4.1, x = 1,purple] -- {a,e,b}};
\tikz[circle, minimum size=2.5mm,inner sep=0pt]
\graph [nodes={fill=blue!70}, empty nodes] { a[y = 0.5,x = -0.5] -- {b[x = 0, y = -0.4],c[y = 0.6],d[y = -0.3]} -- e[y = 0.5, x= 0.5],c--d,a--e, x1[x = 0.6, y = 1.9]--{a,c,d}, b--c, x3[y = 4.1, x = 1,purple] -- {a,e,b}, x4[x = 0.5, y = 4.5, purple] -- {a,x1,c}};
\tikz[circle, minimum size=2.5mm,inner sep=0pt]
\graph [nodes={fill=blue!70}, empty nodes] { a -- {b[x = 0, y = -0.4, purple],c[y = 0.1],d[y = -0.3]} -- e,c--d,a--e, x1[x = 0.6, y = 2.1,purple]--{a,c,d},x2[x = 1.4, y = 3.1,purple]--{e,c,d} , b--c};{$V = 7$}};

\caption{Polyhedral graphs with triangular faces on seven vertices or fewer. Vertices of degree three, coloured red, can be removed without affecting flexibility. Thus, all but three graphs are always rigid.}
\label{SmallPlanarGraphs}
\end{figure}

Let us discuss the question of minimality. Here, a polyhedron is assumed to be of genus zero. Let us say that of two polyhedra, the one with fewer vertices is simpler. Any polyhedron can be triangulated without the addition of vertices, and so it suffices to look for the simplest polyhedron among triangulated polyhedra. Additionally, in a triangulated polyhedron, $3F = 2E$ and $V-E+F = 2$, and so any triangulated polyhedron on $V$ vertices has $F = 2V-4$ faces. Therefore, among triangulated polyhedra, minimising the number of vertices is the same as minimising the number of faces. By applying Theorem 5 from \cite{maksimov_minimal_proof}, which says that any triangulated immersed polyhedron on fewer than eight vertices is rigid, we deduce the following theorem.
\begin{theorem}
The simplest flexible polyhedron embedded in $\R^3$ has eight vertices.
\end{theorem}

Let us briefly discuss the elements of the proof. For further detail, see \cite{maksimov_minimal_proof}. To prove that eight vertices are minimal, we only need to consider the small number of cases that can arise for fewer vertices. Figure \ref{SmallPlanarGraphs} shows every possible example, a listing of planar graphs with triangular faces on seven or fewer vertices. These can be listed by hand following the procedure in \cite{spherical_triangulations}. The next generation follows from the previous one by erecting tetrahedral tents on faces, or splitting an edge by adding a new vertex of degree four. Note that the tetrahedron, $V = 4$, is rigid since three intersecting spheres only meet in two points. Similarly, a tent on a triangular base does not influence the flexibility of a model. Thus, we can remove any such tent and examine the flexibility of the underlying structure. These are the vertices of degree 3, coloured red in the pictures, showing how all but three cases arise as a tetrahedron with some added tents. The three remaining cases are the octahedron (quadrilateral bipyramid), the decahedron (pentagonal bipyramid) and the octahedron with a single tent. In Figure \ref{SmallPlanarGraphs}, these are the first graph for $V = 6$ and the first two graphs for $V = 7$.

Our analysis will be complete once we understand every flexible quadrilateral bipyramid and flexible pentagonal bipyramid. We want to show that the self-intersections of the flexible octahedra cannot be avoided by erecting a single tent, and a flexible pentagonal bipyramid must have self-intersection.

By a theorem of R. Connelly from \cite{connelly_rigidity_theorems_details}, a flexible suspension must have volume zero. Therefore, the flexible octahedron and flexible pentagonal bipyramid must both self-intersect.  The classification of flexible octahedra by R. Bricard in \cite{Bricard} then finishes the argument. From the symmetry of the constructions, type II always involves two edges intersecting, which cannot be solved by a single tent, and type I always involves two pairs of face intersections. The analysis of type III is more algebraic and carried out in full in \cite{maksimov_minimal_proof}.

\section*{Conclusion}

By generalising the reasoning of \cite{Nelson_pentagons}, we have given an alternative proof of the second main result in \cite{Minimal_example}, the existence of flexible polyhedra on eight vertices. The model we obtain shows a significantly larger range of motion than the previously known minimal example. A single action on a Bricard type I octahedron, with the addition of a tent, results in a flexible polyhedron without self-intersection. This shows that the simplest embedded flexible polyhedron is a dodecahedron on eight vertices, but not Steffen's polyhedron as previously believed.

We have not excluded the existence of a polyhedron with less than twelve faces, where some are non-triangular. The known methods of constructing simple, flexible polyhedra do not lend themselves easily to examples with non-triangular faces. A proof of the impossibility of such a polyhedron, or an example with fewer faces, would be of much interest.

Our method of cutting and twisting or cutting and reflecting can be applied to any symmetric quadrilateral lying on a polyhedral surface. As the example shows, this can be a useful way of removing self-intersections of flexible polyhedral surfaces.

\bibliographystyle{alphaurl}
\bibliography{SN_article/bibliography_flexible_polyhedra}
\end{document}